\documentclass{article}
\usepackage{graphicx} 
\usepackage{amsfonts}
\usepackage[utf8]{inputenc}
\usepackage{xcolor}

\usepackage{amssymb}
\usepackage{tikz-cd}
 \usepackage{amscd,amsmath}
\usepackage{amssymb}
\usepackage{amsthm}
\def\Ind#1#2{#1\setbox0=\hbox{$#1x$}\kern\wd0\hbox to 0pt{\hss$#1\mid$\hss}
\lower.9\ht0\hbox to 0pt{\hss$#1\smile$\hss}\kern\wd0}

\def\notind#1#2{#1\setbox0=\hbox{$#1x$}\kern\wd0
\hbox to 0pt{\mathchardef\nn=12854\hss$#1\nn$\kern1.4\wd0\hss}
\hbox to 0pt{\hss$#1\mid$\hss}\lower.9\ht0 \hbox to 0pt{\hss$#1\smile$\hss}\kern\wd0}

\usepackage{amsmath} 
\usepackage{amssymb}
\usepackage{wasysym}
\usepackage{tikz-cd}
\usepackage{amsmath}
\usepackage{subfigure}
\newtheorem{theorem}{Theorem}[section]
\newtheorem{corollary}[theorem]{Corollary}

\newtheorem*{claim}{Claim}

\newtheorem{lemma}[theorem]{Lemma}
\newtheorem{fact}[theorem]{Fact}
\newtheorem{proposition}[theorem]{Proposition}

\newtheorem*{defn}{Definition}

\newtheorem*{thmB,2}{Theorem B, Second version}
\usepackage{amssymb}
\usepackage{tikz-cd}
\usepackage{mathptmx} 
\usepackage{graphicx} 
\usepackage{refcheck}
\usepackage{hyperref}

\title{On $\omega$-categorical groups of finite dimension}
\author{ Moreno Invitti}
\date{}

\usepackage{graphicx} 

\begin{document}

\maketitle
\begin{abstract}
    We prove that a finite-dimensional $\omega$-categorical group is finite-by-abelian-by-finite and that a finite-dimensional $\omega$-categorical ring is virtually finite-by-null.
\end{abstract}
\section*{Introduction}
This article analyzes the structure of $\omega$-categorical groups and rings of finite dimension.
\begin{defn}
 A structure $\mathfrak{M}$ is said to be $\omega$-categorical if its theory $T=T(\mathfrak{M})$ has a unique countable model up to isomorphism.
\end{defn}
Some examples of $\omega$-categorical structures are: dense linear ordered set without endpoints (this is the theory of $(\mathbb{Q},<)$ that is $\omega$-categorical by Cantor's Theorem), random graphs, infinite vector spaces over finite fields, and boolean atomless algebras.\\
The fundamental result that holds for these structures, proved independently by Ryll-Nardzewski \cite{RyllNard59}, Svenonius \cite{svenonius1959no}, and Engeler \cite{engeler1959aquivalenzklassen}, states that an $\omega$-categorical structure has only finitely many parameter-free definable sets, up to equivalence.
\begin{theorem}\label{RyllNard}
Let $\mathfrak{M}$ be a first-order structure. Then, TFAE:
\begin{enumerate}
    \item $\mathfrak{M}$ is $\omega$-categorical;
    \item In any arity, there exist finitely many parameter-free definable sets;
    \item  For any $n<\omega$, there exists only finitely many orbits in $M^n$ for the action of $\operatorname{Aut}_{\emptyset}(\mathfrak{M})$.
\end{enumerate}
\end{theorem}
From Theorem \ref{RyllNard}(3), we can derive the following easy corollary.
\begin{corollary}
    Let $X$ be a definable subset of an $\omega$-categorical structure, $\{c_i\}_{i<n_C}$ a finite set of elements in $X$, $\{R_i\}_{i\in I}$ a family of relations definable over a finite set $R$ and $\{f_i\}_{j\in J}$ a family of functions definable over a finite set $F$. Then, the structure $(X,\{c_i\}_{i\leq n_F},\{R_i\}_{i\in I},\{f_j\}_{j\in J})$ is $\omega$-categorical. 
\end{corollary}
\begin{proof}
Let $D$ the set of definition of $X$ and $n_D=|D|$. Let $n_R=|R|$ and $n_F=|F|$ and $N=n_D+n_C+n_F+n_R$.\\
    It is sufficient to observe that, for any $n<\omega$, the action of $\operatorname{Aut}(\mathfrak{M})$ on the set $M^{n+N}$ has finitely many orbits. In particular, there are finitely many orbits for the subset $(d_1,...,d_{n_D},c_1,...,c_{n_C},f_1,...,f_{n_f},r_1,...,r_{n_R})\times X^n$. Since all the automorphisms of $\mathfrak{M}$ fixing $(d_1,...,d_{n_D},c_1,...,c_{n_C},f_1,...,f_{n_f},r_1,...,r_{n_R})$ are automorphisms of $X$, the proof is completed.
\end{proof}
The study of $\omega$-categorical groups has a long history. The following results are a straightforward application of Theorem \ref{RyllNard}.
\begin{lemma}
    An $\omega$-categorical group is locally finite and of finite exponent.
\end{lemma}
\begin{lemma}\label{ChaSimSer}
    Let $G$ be an $\omega$-categorical group. Then, there exists a chain of subgroups $G=G_0\geq G_1\geq...\geq G_n=\{0\}$ such that 
    \begin{itemize}
        \item $G_i$ is characteristic and definable in $G$ for any $i\leq n$;
        \item $G_i/G_{i+1}$ is characteristically simple for any $i\leq n-1$.
    \end{itemize}
\end{lemma}
The main theorem about $\omega$-categorical groups is the classification of $\omega$-categorical characteristically simple groups obtained by Wilson in \cite{wilson1981algebraic}. 
\begin{fact}\label{CharSimGro}
    For each infinite, countable, $\omega$-categorical, characteristically simple group $H$, one of the following holds: 
    \begin{itemize}
        \item For some prime number $p$, $H$ is an elementary abelian $p$-group;
        \item $H\simeq B(F)$ or $H\simeq B^-(F)$ for some non-abelian finite simple group $F$. $B(F)$ denotes the group of all the continuous functions from the Cantor space $C$ to $F$ (with the discrete topology) while $B^-(F)$ is the subgroup consisting of the functions $f\in B(F)$ such that $f(x_0)=e$ for a fixed element $x_0\in C$.
        \item $H$ is a perfect $p$-group for some prime $p$.
    \end{itemize}
\end{fact}
The important question is whether, under an assumption of "tameness", we can classify $\omega$-categorical groups and $\omega$-categorical rings. In particular, the two big meta-conjectures are the following. 
\begin{theorem}
\begin{enumerate}
    \item A tame $\omega$-categorical group is virtually nilpotent. A tame $\omega$-categorical ring is virtually nilpotent.
    \item A supertame $\omega$-categorical group is finite-by-abelian-by-finite. A supertame $\omega$-categorical ring is finite-by-null-by-finite.
\end{enumerate}
\end{theorem}
These have been proven under several conditions of tameness. If tame is substituted by stable, (1) has been proved by Felgner \cite{felgner1978ℵ0} for groups and by Baldwin and Rose \cite{baldwin1977ℵ0} for rings, while (2) by Bauer, Cherlin, and Macintyre in \cite{baur1979totally}. In the simple case, (2) has been verified by Evans and Wagner in \cite{evans2000supersimple}). (1) holds also for NSOP and NIP theories as proved by Macpherson in \cite{macpherson1988absolutely} and Krupinski in \cite{krupinski2012} (under the assumption of f.s.g. for groups) respectively. Finally, (2) has been proved under the hypothesis of finite burden by Wagner and Dobrowolski in \cite{dobrowolski2020omega}.\\
This article proves (2) under the assumption of finite-dimensionality, an important model-theoretic notion introduced by Wagner in \cite{wagner2020dimensional}. 
In the first section, we introduce finite-dimensional theories and describe the properties of finite-dimensional groups. Moreover, we introduce the notion of dimensional-generic of a group. In the second section, we define bilinear quasi-forms. These, as in \cite{dobrowolski2020omega}, are fundamental in the analysis of $\omega$-categorical groups. In the third section, we define principal indiscernible dimensional-generic sequences. These are used in section 4 to prove the finite-dimensional version of \cite[Theorem 7.5]{dobrowolski2020omega}.
\begin{theorem}\label{VirtTriv}
    A bilinear quasi-form definable in a finite-dimensional $\omega$-categorical structure is virtually almost trivial.
\end{theorem}
In the fifth section, we apply Theorem \ref{VirtTriv} to prove the following result.
\begin{theorem}
    An $\omega$-categorical group of finite dimension is finite-by-abelian-by-finite, and an $\omega$-categorical ring of finite dimension is virtually null-by-finite.
\end{theorem} 
\section{Finite-dimensional groups}
We define finite-dimensional theories.
\begin{defn}
    A theory $T$ is \emph{finite-dimensional} if there exists a function $\operatorname{dim}$ from the class of all the interpretable subsets in any model $\mathcal{M}$ of $T$ into $\omega\cup\{-\infty\}$ such that for any $\phi(x,y)$ formula, $X,Y$ interpretable sets in $T$ and $f$ interpretable function from $X$ to $Y$, then
\begin{itemize}
    \item If $a,a'$ have the same type over $\emptyset$, $\operatorname{dim}(\phi(x,a))=\operatorname{dim}(\phi(x,a'))$;
    \item $\operatorname{dim}(\emptyset)=-\infty$ and $\operatorname{dim}(X)=0$ if and only if $X$ is finite;
    \item $\operatorname{dim}(X\cup Y)=\max\{\operatorname{dim}(X),\operatorname{dim}(Y)\}$;
    \item If $\operatorname{dim}(f^{-1}(y))\geq k$ for any $y\in Y$, then $\operatorname{dim}(X)\geq \operatorname{dim}(Y)+k$;
    \item If $\operatorname{dim}(f^{-1}(y))\leq k$ for any $y\in Y$, then $\operatorname{dim}(X)\leq \operatorname{dim}(Y)+k$.
\end{itemize}
\end{defn}
Examples of finite-dimensional theories are superstable theories of finite Lascar rank, supersimple theories of finite Lascar rank, and $o$-minimal theories.\\
We now focus our attention on finite-dimensional groups.  An important property of finite-dimensional groups is the $\omega$-DCC. 
\begin{defn}
    Let $G$ be a definable group. Then, $G$ has the \emph{$\omega$-DCC} property if there exists no infinite strictly descending chain of definable subgroups $\{G_i\}_{i<\omega}$ such that $|G_i:G_{i+1}|>\omega$.
\end{defn}
Groups definable in a finite-dimensional theory respect the $\omega$-DCC (Corollary 2.4 of \cite{wagner2020dimensional}). Moreover, finite-dimensional groups have the ucc, as a consequence of the following fundamental Lemma \cite[Lemma 1.1]{invitti2025End}.
\begin{lemma}\label{boundedind}
    Let $G$ be a definable group of finite dimension and $\{H_i\}_{i\in I}$ a family of uniformly definable subgroups. Then, there exist $n,d<\omega$ such that there is no $J=\{j_1,...,j_n\}\subseteq I$ of cardinality $n$ with $|\bigcap_{i=1}^kH_{j_i}:\bigcap_{i=1}^{k+1} H_{j_i}|\geq d$ for any $k\leq n-1$.
\end{lemma}
If we apply these results to the centralisers of elements in $G$, we obtain that a finite-dimensional definable group has the hereditarily $\widetilde{\mathfrak{M}}_c$-property.
\begin{defn}
    A group $G$ is \emph{hereditarily-$\widetilde{\mathfrak{M}}_c$} if, for any definable subgroups $H,N$ such that $N$ is normalised by $H$, there exist natural numbers $n_{HN}$ and $d_{HN}$ such that any sequence of centralisers 
    $$C_H(a_0/N)\geq C_H(a_0,a_1/N)\geq...\geq C_H(a_0,a_1,...,a_n/N)\geq...$$
    with $|C_H(a_0,...,a_n/N):C_H(a_0,...,a_{n+1}/N)|\geq d_{HN}$ has length at most $n_{HN}$.
\end{defn}
The notions of almost containment and commensurability are fundamental in the study of groups definable in finite-dimensional theories.
\begin{defn}
\begin{itemize}
    \item  Let $G$ be a group and $A, B\leq G$ subgroups. $A$ \emph{almost contains} $B$, denoted $A\apprge B$, if $|B:B\cap A|$ is finite.
    \item If either $A\apprle B$ or $B\apprle A$, we say that $A$ and $B$ are \emph{comparable};
    \item  If $A\apprle B$ and $B\apprle A$, $A$ and $B$ are \emph{commensurable}, denoted $A\sim B$. $\sim$ is clearly an equivalent relation on the class of subgroups of $G$.
    \item Given a family of subgroups $\{G_i\}_{i\in I}$ in $G$, the family $\{G_i\}_{i\in I}$ is \emph{uniformly commensurable} if there exists $n<\omega$ such that $|G_i/G_i\cap G_j|\leq n$ for every $i,j\in I$.
\end{itemize}
  \end{defn}
We also define the notion of almost direct sum.
\begin{defn}
    Let $G$ be a group and $H,K$ subgroups of $G$ that normalize each other. The sum $H+K$ is \emph{almost direct}, denoted by $H\widetilde{\oplus} K$, if $H\cap K$ is finite.
\end{defn}
A consequence of the hereditarily $\widetilde{\mathfrak{M}}_c$-condition is that the almost centralisers are definable.
\begin{defn}
    Let $G$ be a group, $K,H\leq G$ and $N$ a subgroup of $G$ normalized by both $H$ and $K$. The \emph{almost centraliser} in $K$ of $H/N$, denoted $\widetilde{C}_K(H/N)$ is the subgroup
    $$\{k\in K:\ C_H(k/N)\apprge H\}.$$
\end{defn}
If we assume $N=0$ and $K=H=G$, we have the notion of almost center.
\begin{defn}
    Let $G$ be a group. We define iteratively the $n$-\emph{almost center} of $G$, denoted by $\widetilde{Z}^n(G)$, as follows:
    \begin{itemize}
        \item $\widetilde{Z}^1(G)=\widetilde{C}_G(G)$;
        \item $\widetilde{Z}^{n+1}(G)$ is such that $\widetilde{Z}(G/\widetilde{Z}^n(G))=\widetilde{Z}^{n+1}(G)/\widetilde{Z}^n(G)$.
    \end{itemize}
\end{defn}
The following Lemma is proven in \cite[Proposition 3.3]{hempel2020almost}.
\begin{lemma}
    Let $G$ be a hereditarily $\widetilde{\mathfrak{M}}_c$-group and $K,H,N$ as before. If $K,H,N$ are definable, $\widetilde{C}_K(H/N)$ is definable.
\end{lemma}
An easy result for the almost center is the following.
\begin{lemma}\label{Zfin}
    Let $G$ be a group such that $\widetilde{Z}(G)$ is finite. Then, $\widetilde{Z}(G/\widetilde{Z}(G))=\widetilde{Z}(G)$.
\end{lemma}
\begin{proof}
    Let $g\in \widetilde{Z}(G/\widetilde{Z}(G))$. Then, $[g,G/\widetilde{Z}(G)]+\widetilde{Z}(G)$ is finite. Since $\widetilde{Z}(G)$ is finite, $[g,G]$ is finite \hbox{i.e.} $g\in \widetilde{Z}(G)$.
\end{proof}
Another important property of almost centralisers for definable groups is symmetry \cite[Theorem 2.10]{hempel2020almost}.
\begin{lemma}\label{sym}
    Let $G$ be a definable group and $H,K,N$ definable subgroups such that both $H$ and $K$ normalize $N$. Then, $H\apprle \widetilde{C}_G(K/N)$ iff $K\apprle \widetilde{C}_G(H/N)$.
\end{lemma}
In the last section of this article, we will use a fundamental theorem of linearization for the action of an almost abelian group definable in a finite-dimensional theory.
\begin{defn}
    Let $G$ be a group. $G$ is \emph{almost abelian} if $\widetilde{Z}(G)=G$.
\end{defn}
We introduce the notion of absolutely minimal $G$-group.
\begin{defn}
    Let $G$ be a definable group acting on the definable group $A$. $A$ is \emph{absolutely $G$-minimal} if, for any definable subgroup $H$ of finite index in $G$, $A$ is $H$-minimal. 
\end{defn}
We introduce the notion of almost centralisers of an action.
\begin{defn}
    Let $G$ be a group acting on a group $A$. The \emph{almost centraliser of the action in $G$}, denoted by $\widetilde{C}_G(A)$, is the subgroup $\{g\in G:\ C_A(g)\sim A\}$. The \emph{almost centraliser of the action in $A$}, denoted by $\widetilde{C}_A(G)$, is the subgroup $\{a\in A:\ C_G(a)\sim G\}$.
\end{defn}
It follows from Lemma \ref{boundedind} that the two almost centralisers are definable. We introduce almost trivial actions.
\begin{defn}
    Let $G$ be a group acting on a group $A$. The action of $G$ on $A$ is \emph{almost trivial} if either $\widetilde{C}_G(A)$ is of finite index in $G$ or $\widetilde{C}_A(G)$ is of finite index in $A$.
\end{defn}
In the finite-dimensional case, these two conditions are equivalent by \cite[Lemma 12.4]{invitti2025End}. The following result is fundamental in the analysis of almost trivial actions.
\begin{lemma}\label{lemma 4}
    Let $G$ be a definable almost abelian group acting on the definable almost abelian group $A$. Then, 
    $$[\widetilde{C}_G(A),\widetilde{C}_A(G)]=\langle [g,a]=ga-a:\ g\in \widetilde{C}_G(A),a\in \widetilde{C}_A(G)\rangle$$
    is finite.
\end{lemma}
\begin{proof}
    By Lemma \ref{boundedind}, the almost centralisers $H=\widetilde{C}_G(A),B=\widetilde{C}_A(G)$ are definable subgroups such that $B$ is $G$-invariant and $H$ is $G$-normal. Since $A$ is almost abelian, $B'$ is finite and characteristic, and $H$ acts on the abelian group $B/B'$. Since $B$ is definable and every element has finite $H$-orbit, there exists $b\in B$ such that $Hb$ is of maximal cardinality. Let $h_1,...,h_n$ be a left transversal of $C_H(b)$ in $H$. Since every element $h\in H$ has centraliser in $A$ of finite index, $C=C_B(h_1,...,h_n)$ is of finite index in $B$. Let $b_1,...,b_m$ be a transversal of $C$ in $B$. Denote by $F$ the $H$-invariant finitely generated subgroup $\langle Hb,Hb_i,...,Hb_m\rangle$. Clearly, $F$ has countable cardinality. Moreover, by definition, $B=F+C$. On the other hand, let $D=\{[h,b]:\ h\in H,b\in B\}$. $D$ is clearly a definable subset of $B$ and therefore either finite or unbounded. We verify that $D\leq F$. Given $b\in B$, there exist $c\in C$ and $f\in F$ such that $b=c+f$. Therefore, $[h,b]=h(c+f)-c-f=h(c)-c+h(f)-f$. Since $h(f)\in F$, $[h,b]$ is contained in $F$ iff $h(c)-c\in F$. Define $w=c+a$, then $h_i(w)=h_i(c)+h_i(b)$ but $c\in C$ and so this is equal to $c+h_i(b)$. Since $h_i(b)\not= h_j(b)$ when $i\not=j$ and by maximality of the cardinality of the orbit of $b$, the orbit of $w$ is $c+g_i(b)$ for $i=1,...,n$. Therefore, for any $h\in H$, $h(w)=c+h_i(b)$ for a certain $i$. Consequently,
    $$h(c)-c=h(w)-h(b)-c=c+h_i(b)-h_j(b)-c=h_i(b)-h_j(b)\leq F.$$
    Being $D$ definable and not unbounded, it is finite. Therefore, the subgroup $[H,B]$ is finitely generated and, being $B$ abelian, it is finite iff the order of any generator is finite. But $n(g-1)(b)=g-1(nb)\leq D$ that is finite so also $o(g-1(b))$ is finite for every $g\in H$ and $b\in B$. Since $H'$ is finite, this proves the theorem.
\end{proof}
Finally, if $G$ and $A$ are two almost abelian groups such that $G$ acts on $A$ not almost trivially and $A$ is absolutely $G$-minimal, then the action can be linearized \cite[Theorem 12.6]{invitti2025End}.
\begin{theorem}\label{LinInv}
    Let $G$ be a definable almost abelian group acting on a definable infinite almost abelian group $A$. Assume that:
    \begin{itemize}
        \item $A$ is absolutely $G$-minimal;
        \item the action is not almost trivial.
    \end{itemize}
    Then, $G/\widetilde{C}_G(A)$ is abelian and there exists a definable field $K$ such that $G/\widetilde{C}_G(A)$ definably embeds in $K^{\times}$ and $A/\widetilde{C}_A(G)$ is isomorphic to $K^{+}$.
\end{theorem}
\subsection{Dimensional-generics}
We introduce the notion of dimensional-generics of a definable group of finite dimension. We also recall the notion of subgroup-generic used in \cite{dobrowolski2020omega}.
\begin{defn}
Given $G$ a group definable in a theory $T$, a definable subset $X$ of $G$ is:
    \begin{itemize}
        \item \emph{subgroup-generic} if it is not contained in the union of finitely many translates of interpretable subgroups of infinite index;
        \item \emph{dimensional-generic} if $\dim(G)=\dim(X)$ with $T$ a finite-dimensional theory.
        
    \end{itemize}
    A (possibly partial) type $p$ is $\ast$-generic if every $\phi\in p$ is $\ast$-generic (with $\ast$ equal to either subgroup or dimensional).
\end{defn}
We prove the existence of dimensional-generics.
\begin{lemma}\label{ExistGen}
Let $G$ be a definable finite-dimensional group and $A$ a set of parameters. Then, there always exists a dimensional-generic type $p\in S(A)$. Moreover, we can construct infinite indiscernible sequences $\{a_i\in G\}_{i\in I}$ over $A$ such that $tp(a_i/\{a_j\}_{j<i}\cup A)$ is dimensional-generic.
\end{lemma}
\begin{proof}
Assume that $G$ is $\emptyset$-definable.\\
For the first part, let $q$ be the partial type over $A$ given by $\{X\subseteq G:\ \dim(G-X)<\dim(G)\}$. We verify that it is finitely consistent. Assume not, then there exists a finite subfamily $\{X_i\}_{i\leq n}$ of $q$ such that $\bigcap_{i\leq n} X_i$ is empty. This implies that $G-\bigcap_{i\leq n}X_i=\bigcup_{i\leq n} G-X_i=G$. Therefore, $\dim(G)=\dim(\bigcup_{i\leq n} G-X_i)=\max_{i\leq n} \dim(G-X_i)$ by union. By hypothesis, $\dim(G)=\max_{i\leq n} \dim(G-X_i)<\dim(G)$, a contradiction. Therefore, let $p$ be a completion of $q$ in $S_A(G)$. Then, this type is dimensional-generic. Assume not, then there exists an $A$-definable set $X$ in $p$ such that $\dim(X)<\dim(G)$. By definition of $q$, $G-X\in q\subseteq p$ and so $X,G-X\in p$, contradicting the finite satisfability of $p$.\\
For the second part, apply iteratively the first part to obtain an infinite sequence $\{a_i\}_{i\in I}$ of elements in $G$ such that $tp(a_i/A\cup\{a_j\}_{j<i}\}$ is dimensional-generic. This sequence is also $A$-dimensional independent since $\dim(tp(a_i/A\cup \{a_j\}_{j<i})=\dim(G)\geq \dim(tp(a_i/A))$ since $a_i\in G$. The conclusion follows by Ramsey's Theorem.
\end{proof}
We verify that dimensional-generics are subgroup-generics.
\begin{lemma}\label{dimgen}
 A dimensional-generic is subgroup-generic.
\end{lemma}
\begin{proof}
Let $p$ be a dimensional-generic for a definable group $G$ of finite dimension. Assume, for a contradiction, that $p$ is not subgroup-generic. Then, by definition, there exists $\phi(x)\in p$ such that $\phi(G)$ is contained in a translate of an interpretable subgroup $H$ of infinite index in $G$. By fibration, $\dim(H)<\dim(G)$. Therefore, by invariance, $\dim(\phi(G))\leq \dim(H)<\dim(G)$, contradicting the dimensional-genericity. 
\end{proof}
\section{Bilinear quasi-forms}
We introduce bilinear quasi-forms, which are a fundamental instrument in the analysis of $\omega$-categorical groups and rings of finite dimension.
\begin{defn}
A \emph{bilinear quasi-form} $\lambda$ from $G,H$ to $K$ abelian groups is a partial function from $G\times H$ to $K$ such that, for every $g\in G$ and $h\in H$, the partial functions 
$$\lambda_g:=\lambda(g,\_):h\in \operatorname{dom}(\lambda(g,\_))\to \lambda(g,h)\in K$$
and 
$$\lambda'_h:=\lambda(\_,h):g\in \operatorname{dom}(\lambda(\_,h))\to \lambda(g,h)\in K$$
are partial endomorphisms. We denote $\operatorname{ann}_H(g)=\{h\in H:\ \lambda(g,h)=0_K\}$ and similarly $\operatorname{ann}_G(h)$ for $h\in H$.
\end{defn}
We define the almost annihilator of a bilinear quasi-form.
\begin{defn}
Let $\lambda$ be a bilinear quasi-form from $G\times H$ to $K$ and let $A\leq G$ and $B\leq H$. We denote by $\widetilde{\operatorname{ann}}_H(A)$ the \emph{almost annihilator of $A$ over $H$ with respect to $\lambda$} defined as
$$\widetilde{\operatorname{ann}}_H(A):=\{h\in H:\ \lambda(A,h)\text{ is finite}\}=\{h\in H:\ \operatorname{ann}_G(h)\apprge A\}.$$
Similarly, we define the \emph{almost annihilator of $B$ over $G$ with respect to $\lambda$} the subgroup
$$\widetilde{\operatorname{ann}}_G(B):=\{g\in G:\ \lambda(g,B)\text{ is finite}\}=\{g\in G:\ \operatorname{ann}_H(g)\apprge B\}.$$
If the quasi-form is clear, we call them almost annihilator of $A$ over $H$ and almost annihilator of $B$ over $G$, respectively. 
\end{defn}
It follows from Lemma \ref{boundedind} that, if $G,H,A,B$ and $\phi$ are definable in a finite-dimensional structure, the almost annihilators are definable.
\begin{lemma}\label{definability}
    Assume we are working in a finite-dimensional theory $T$. Let $\phi$ be a definable bilinear quasi-form from $G\times H$ to $K$ and $A,B$ definable subgroups of $G,H$ respectively. Then, $\widetilde{\operatorname{ann}}_H(A)$ and $\widetilde{\operatorname{ann}}_G(B)$ are definable.
\end{lemma}
\begin{proof}
    From Lemma \ref{boundedind} applied the uniformly definable family of subgroups $\{\operatorname{ann}_G(h)\}_{h\in \widetilde{\operatorname{ann}}_H(A)}$, there exists $N\in \mathbb{N}$ such that $N>|G:\operatorname{ann}_G(h)|$. This implies that $\widetilde{\operatorname{ann}}_H(A)=\{h\in H:\ |G:\operatorname{ann}_G(h)|\leq n\}$ is definable. The conclusion for $\widetilde{\operatorname{ann}}_G(B)$ follows similarly.
\end{proof}
We introduce the notions of triviality, almost triviality, and virtual almost triviality.
\begin{defn}
    A bilinear quasi-form $\lambda$ is \emph{trivial} if $\lambda(g,h)=0$ for all $g\in G,h\in H$. It is \emph{almost trivial} if $\lambda(G,H)\leq S$ with $S$ a finite subgroup of $K$. $\lambda$ is \emph{virtually almost trivial} if there exists $G_0\leq G$ and $H_0\leq H$ subgroups of finite index such that $\lambda_{|G_0\times H_0}$ is almost trivial.
\end{defn}
We have a characterization of virtually almost trivial bilinear quasi-form given by \cite[Corollary 5.11]{dobrowolski2020omega}.
\begin{lemma}\label{almosttrivial}
    Let $\lambda:G\times H\to K$ be a definable bilinear quasi-form. Then, TFAE:
    \begin{itemize}
        \item $\lambda$ is virtually almost trivial;
        \item $\widetilde{\operatorname{ann}}_G(H)\apprge G$;
        \item $\widetilde{\operatorname{ann}}_H(G)\apprge H$.
    \end{itemize}
\end{lemma}
From now on, we assume that the bilinear quasi-forms are definable in a finite-dimensional theory. We introduce the notion of reduced dimension of a bilinear quasi-form.
\begin{defn}
    Let $\lambda:G\times H\to K$ be a definable bilinear quasi-form of finite dimension. Then, we define
    \begin{itemize}
        \item the \emph{reduced dimension of $G$}, denoted as $\operatorname{rdim}_{\lambda}(G)$, as $\max\{\operatorname{dim}(G/\operatorname{ann}_G(\overline{h}))\}_{\overline{h} \text{ finite tuple in } H}$;
        \item the \emph{reduced dimension of $H$} denoted as $\operatorname{rdim}_{\lambda}(H)$ as $\max\{\operatorname{dim}(H/\operatorname{ann}_H(\overline{g}))\}_{\overline{g} \text{ finite tuple in } G}$;
        \item the \emph{reduced dimension of $\lambda$} denoted as $\operatorname{rdim}(\lambda)$ as $\operatorname{rdim}_{\lambda}(G)+\operatorname{rdim}_{\lambda}(H)+\operatorname{dim}(K)$.
    \end{itemize}
\end{defn}
We also introduce the notion of induced bilinear quasi-form.
\begin{defn}
    A bilinear quasi-form $\overline{\lambda}:A\times B\to D/C$ is \emph{induced from} $\lambda:G\times H\to K$ if $A\leq G,B\leq H,C,D\leq K$ and $\pi\circ \lambda:\lambda^{-1}(D)\cap A\times B\to D/C$ coincides with $\overline{\lambda}$ where $\pi$ is the projection from $D$ to $D/C$.
\end{defn}
Observe that, in order to be $\overline{\lambda}$ well-defined, $A\times B$ must be almost contained in $\lambda^{-1}(D)$.\\
We verify that, given $\lambda$ a bilinear quasi-form, the reduced dimension of any induced bilinear quasi-form is less than $\operatorname{rdim}(\lambda)$.
\begin{lemma}\label{induced}
    if $\overline{\lambda}:A\times B\rightarrow D/C$ is induced from $\lambda$, then $\operatorname{rdim}_{\overline{\lambda}}(B)\leq \operatorname{rdim}_{\lambda}(H)$ and $\operatorname{dim}(D/C)\leq \operatorname{dim}(K)$. In particular, $\operatorname{rdim}(\overline{\lambda})\leq \operatorname{rdim}(\lambda)$.
\end{lemma}
\begin{proof}
    The latter conclusion follows easily from the previous one.\\
    For the first point, it is sufficient to prove that 
    $$\operatorname{max}\{\operatorname{dim}(H/\operatorname{Ker}(\lambda_g))\}_{g\in G}\geq \operatorname{max}\{\operatorname{dim}(B/\operatorname{Ker}(\overline{\lambda}_a))\}_{a\in A}.$$ 
    For every $\overline{a}\in A$, denote $\lambda^C_a$ the function that sends $h\in H$ to $\lambda_a(h)+C$. Then, 
    $$\operatorname{dim}(H/\ker(\lambda_a))\geq \operatorname{dim}(H/\ker(\lambda_a^C))\geq \operatorname{\operatorname{dim}(A/\ker(\lambda_a^C\cap A))}=\operatorname{dim}(A/\ker(\overline{\lambda}_a)).$$
    This completes the proof. 
\end{proof}
Finally, we recall \cite[Lemma 4.6]{dobrowolski2020omega}.
\begin{lemma}\label{Lemm 4.6}
    Let $G$ be an $\omega$-categorical abelian group and $f$ a definable quasi-endomorphism of $G$. Then, there is $n<\omega$ such that $\mathrm{im}(f^{{\circ}n})\widetilde{\oplus}\operatorname{ker}(f^{\circ n})\sim G$.
\end{lemma}
\section{Principal indiscernible sequences}
We introduce principal dimensional-generic indiscernible sequences. As proven in Lemma \ref{dimgen}, any dimensional-generic is subgroup-generic. Therefore, these sequences are particular examples of principal subgroup-generic sequences introduced in \cite{dobrowolski2020omega}. On the other hand, the existence of principal indiscernibles finite-dimensional sequences will substantially simplify the proof of Theorem \ref{VirtTriv}.
\begin{defn}
    Let $G$ be an infinite definable group and $A$ a set of parameters. An $A$-indiscernible sequence $\langle (g_i,\overline{a}_i);i\in I\rangle$ with $g_i\in G$ for each $i\in I$ is \emph{principal indiscernible} if, for any $i\in I$ and $A_i=A\cup\{g_j,\overline{a}_j:j\not=i\}$, whenever $C$ is an $A_i$-definable coset of some subgroup $H$ and $g_i\in C$, then $g_i\in H^0_{A_i}$. A sequence $\{g_i\}_{i\in I}$ is \emph{principal indiscernible dimensional-generic} if it is principal, indiscernible, and dimensional-generic.
\end{defn}
We modify slightly \cite[Proposition 6.4]{dobrowolski2020omega}. 
\begin{proposition}\label{Prinseq}
   Let $\langle (g_i,\overline{a}_i):i\in \mathbb{Q}\cup(\mathbb{Q}+\epsilon)\rangle$, with $\epsilon$ an infinitesimal, be an $A$-indiscernible sequence. Then, the sequence $\langle (h_i,\overline{a}_i\overline{a}_{i+\epsilon}):i\in \mathbb{Q}\rangle$ with $h_i=g^{-1}_{i+\epsilon}g_i$ is principal indiscernible. If, moreover, $\langle (g_i,\overline{a}_i):i\in \mathbb{Q}\cup(\mathbb{Q}+\epsilon)\rangle$ is dimensional-generic, then also $\langle (h_i,\overline{a}_i\overline{a}_{i+\epsilon}):i\in \mathbb{Q}\rangle$ is dimensional-generic. In particular, principal indiscernible dimensional-generic sequences exist.
\end{proposition}
\begin{proof}
    The first part follows from \cite[Proposition 6.4]{dobrowolski2020omega}. The second part by the fact that if $\operatorname{tp}(g_{i+\epsilon}/A,\{g_j,\overline{a}_j:j\leq i\})$ is dimensional-generic, also $\operatorname{tp}(g^{-1}_{i+\epsilon}/A,\{g_j,\overline{a}_j:j\leq i\})$ and $\operatorname{tp}(g^{-1}_{i+\epsilon}g_i/A,\{g_j,\overline{a}_j:j\leq i\})$ are dimensional-generic. From Lemma \ref{ExistGen}, there exists an indiscernible dimensional independent sequence. Therefore, applying the construction of the first part, we may obtain a principal indiscernible dimensional-generic sequence. This completes the proof of the Lemma.
\end{proof}
\section{Virtual almost triviality}
We now prove Theorem \ref{VirtTriv}. By Lemma \ref{almosttrivial}, it is sufficient to show that one of the two almost annihilators is of finite index in $G$ or $H$, respectively.\\
We prove Theorem \ref{VirtTriv} by induction on the reduced dimension of the bilinear quasi-form $\lambda$. Since for dimension $0$, the conclusion is clear, we may assume that $\operatorname{rdim}(\lambda)=N$ and that any bilinear quasi-form of reduced dimension strictly less than $N$ is virtually almost trivial.\\
We need some preliminary lemmas.
\begin{lemma}\label{Comparability}
Let $\lambda:G\times H\to K$ be a definable bilinear quasi-form, and that $G,H,K$ are $\omega$-categorical groups of finite dimension.\\
Assume that any induced quasi-form of strictly smaller reduced dimension is virtually almost trivial. Let $\langle y_i\rangle_{i\in \mathbb{Q}}$ be a principal indiscernible sequence in $G$ over $\emptyset$. Then, $\mathrm{im}(\lambda_{y_i})$ and $\mathrm{im}(\lambda_{y_j})$ are comparable. The same holds for the annihilators.
\end{lemma}
\begin{proof}
    Let $C=\sum_{i=0}^n \mathrm{im}(\lambda_{y_{l_i}})$ be a sum of maximal dimension among the finite sums of images of $\lambda_{y}$ for $y\in \langle y_i\rangle_{i\in \mathbb{Q}}$. Moreover, we may assume that, for any $j\leq n$, $\dim(\sum_{i=0,i\not=j}^n \mathrm{im}(\lambda_{y_{l_i}}))<\dim(\sum_{i=0}^n \mathrm{im}(\lambda_{y_{l_i}}))$. By maximality of the dimension, for every $y\in \langle y_i\rangle_{i\in \mathbb{Q}}$, $\mathrm{im}(\lambda_y)$ is almost contained in the previous sum.\\
    We consider the induced bilinear quasi-form $\overline{\lambda}:G\times H\to K/C$ and denote by $A$ the almost annihilator of $\overline{\lambda}$ in $G$ and by $B$ the subgroup $\{h\in H:\ \lambda_h(A) \text{ is finite}\}$. $A,B$ are definable by Lemma \ref{definability}, and moreover, $H$ is almost contained in $B$ by Lemma \ref{sym}. By definition, $\lambda_a[B]\leq \lambda_a[H]\apprle C$ for any $a\in A$ and similarly for $h\in B$. The restriction 
    $$\lambda:A\times B\to C$$
    is an induced $\{y_{l_i}\}_{i\in \mathbb{Q}}$ form as is 
    $$\lambda_i:A\times B\to C/\mathrm{im}(\lambda_i)\cap C.$$
    Each of the $\lambda_i$ has dimension strictly less than $\operatorname{dim}(\lambda)$ (by minimality of $n$) and so it must be virtually almost trivial. Therefore, the almost annihilator is a $\{y_i\}_{i\leq n}$-definable subgroup of finite index in $A$ \hbox{i.e.} it contains $A^0_{\{y_{l_i}:i\leq n\}}$. Since $y$ is principal, it must be contained in the connected component and so $\mathrm{im}(\lambda_y)\apprle \mathrm{im}(\lambda_{y_i})$. By principal indiscernibility, the conclusion follows.\\
    For the annihilators, take an intersection $B:=\bigcap_{i\leq n} \operatorname{ann}(\lambda_{y_{l_i}})$ of minimal dimension possible and with $n$ minimal among the intersections of minimal dimension.\\
    Let $A:=\{g\in G:\ B\apprle \operatorname{ann}_H(g)\}$ and consider the following quasi bilinear form
    $$\lambda_i:A\times \operatorname{ann}_H(y_{l_i})\to K.$$
    We show that the reduced dimension of $\lambda_i$ is strictly less than $\operatorname{rdim}(\lambda)$. It is sufficient to prove that $\operatorname{rdim}_{\lambda_i}(\operatorname{ann}_H(y_{l_i}))<\operatorname{rdim}_{\lambda}(H)$. Taken $\overline{g}\in G$, then $\operatorname{ann}_H(\overline{g})\apprge B$ and so the dimension of $\operatorname{ann}_H(y_{l_i})/\operatorname{ann}_{\operatorname{ann}_H(y_{l_i})}(\overline{g})$ is less than $\dim(\operatorname{ann}_h(y_{l_i})/B)$. By minimality, $\operatorname{dim}(\bigcap_{j=1,j\not=i}\operatorname{ann}_H(y_j)/B)>0$ and so
    $$\operatorname{dim}(\operatorname{ann}_H(y_{l_i})/B)<\operatorname{dim}(\operatorname{ann}_H(y_{l_i})/B)+\operatorname{dim}\big(\bigcap_{j=1,j\not=i}\operatorname{ann}_H(y_j)/B\big)\leq \operatorname{dim}(H/B)$$
    that is less than the reduced dimension of $H$. Therefore, $\lambda_i$ is virtually almost trivial. Consequently, the almost annihilator is a subgroup of finite index in $A$ and so it contains the connected component. This implies that $\operatorname{ann}_H(y_{l_i})\apprle \operatorname{ann}_H(y_{l_j})$.
\end{proof}
\begin{lemma}\label{NilpoIsom}
    Suppose that $\lambda$ is definable in a finite-dimensional theory and that any induced bilinear quasi-form of strictly smaller dimension is virtually almost trivial. Let $\langle y_i\rangle_{i\in \mathbb{Q}}$ be a principal dimensional-generic sequence. Then, for any $j\in \mathbb{Q}$, any definable quasi-endomorphism $f$ of $H/\operatorname{ann}_H(y_j)$ is either invertible or nilpotent.
\end{lemma}
\begin{proof}
    By Lemma \ref{Lemm 4.6}, there exists $n<\omega$ such that $G\apprle\mathrm{im}(f^n)\widetilde{\oplus}\operatorname{ker}(f^n)$. The group 
    $$A:=\{g\in G:\ \operatorname{ann}_H(y_j)\apprle \operatorname{ann}_H(g)\}$$
    is a definable subgroup by Lemma \ref{boundedind}. Denote $B_1,B_2$ the preimages in $H$ of $\mathrm{im}(f^n)$ and $\operatorname{ker}(f^n)$ respectively. Then $B_1\cap B_2$ is a finite extension of $\operatorname{ann}_H(y_j)$. If $f$ is neither invertible nor nilpotent, then both the $B_i$ are infinite and the restriction of $\lambda$ to $A\times B_i$ has strictly smaller dimension. Therefore, it is virtually almost trivial and the same for the bilinear quasi-form 
    $$\lambda':A\times H\to K$$
    given by the sum of the two. By Lemma \ref{Comparability}, $y_i\in A$ for either $i<j$ or $j>i$. Therefore, $y_i\in \operatorname{ann}_G(H)$ and the latter must be of finite index since $y_i$ is dimensional-generic. Therefore, by Lemma \ref{almosttrivial}, $\lambda$ is virtually almost trivial.
\end{proof}
We start the proof of Theorem \ref{VirtTriv}.
\begin{proof}
    Let $\lambda$ be a counterexample of Theorem \ref{VirtTriv} of minimal dimension. Up to add finitely many parameters, we can assume that everything is $\emptyset$-definable. Let $\langle (x_i,x'_i):i\in \mathbb{Q}\cup\mathbb{Q}+\epsilon\rangle$ be a dimensional-generic sequence, that exists by Lemma \ref{ExistGen}. Define $y_i=x_i-x_{i+\epsilon}$ and $y'_i=x'_i-x'_{i+\epsilon}$. Then, by Lemma \ref{Prinseq}, $\langle (y_i,y'_i):i\in \mathbb{Q}\rangle$ is a principal indiscernible dimensional-generic sequence.
    \begin{claim}
        For $i<j$, $\operatorname{ann}_H(y_j)\apprle \operatorname{ann}_H(y_i)$ and $\mathrm{im}(\lambda_i)\apprle \mathrm{im}(\lambda_j)$.
    \end{claim}
    \begin{proof}
        The second is clearly a consequence of the first and of Lemma \ref{Comparability}.\\
        $\widetilde{\operatorname{ann}}_H(G)$ is definable and the indexes $|G:\operatorname{ann}_G(h)|$ takes only finitely many values. Let $n$ be the maximal one and let $h$ realize a generic type for $\widetilde{\operatorname{ann}}_H(G)$. Then, taken $x_0,x_1,...,x_n$, there exist $i\not=j$ such that $x_i-x_j\in \operatorname{ann}_G(h)$ and so $h\in \operatorname{ann}_H(x_i-x_j)$. By dimensional-genericity of $h$, the subgroup $\operatorname{ann}_H(x_i-x_j)\cap \widetilde{\operatorname{ann}}_H(G)$ is of finite index in $\widetilde{\operatorname{ann}}_H(G)$ \hbox{i.e.} $x_i-x_j\in \widetilde{\operatorname{ann}}_G(\widetilde{\operatorname{ann}}_H(G))$. By indiscernibility, $\operatorname{ann}_H(y_0)\apprge \widetilde{\operatorname{ann}}_H(G)$. Suppose $\operatorname{ann}_H(y_0)\apprle \operatorname{ann}_H(y_1)$. Then $y_1\in \widetilde{\operatorname{ann}}_G(\operatorname{ann}_G(y_0))$. By dimensional-genericity, $\widetilde{\operatorname{ann}}_G(\operatorname{ann}_G(y_0))$ should be of finite index. Consequently, $\widetilde{\operatorname{ann}}_H(G)\apprge \widetilde{\operatorname{ann}}_H(y_0)$ and the same must hold for every $y_i$ for $i\in \mathbb{Q}$ by indiscernibility. This implies that $\operatorname{ann}_H(y_0)\sim \widetilde{\operatorname{ann}}_H(G)\sim \operatorname{ann}_H(y_1)$. If $\operatorname{ann}_H(y_1)\apprle \operatorname{ann}_H(y_0)$, the conclusion follows.\\
    \end{proof}
    
   We define $\lambda_{y_1,y_j}=\lambda_{y_j}^{-1}\circ \lambda_{y_i}$ to be the definable additive relation in $H\times H$ given by 
   $$\{(h,h'):\ \lambda(y_i,h)=\lambda(y_j,h')\}.$$
   Clearly, the kernel of this relation is $\operatorname{ann}_H(y_i)$ and $\operatorname{kat}(\lambda_{y_i,y_j})=\operatorname{ann}_H(y_j)$.
   By Claim $1$, this defines a quasi-homomorphism $\overline{\lambda}_{y_i,y_j}$ of $H/\operatorname{ann}_H(y_j)$. For the definition of additive relation and quasi-homomorphism see section 3 of \cite{dobrowolski2020omega}.
    \begin{claim}
        For $i\not=j$, $\operatorname{ann}_H(y_i)\not\sim \operatorname{ann}_H(y_j)$.
    \end{claim}
    \begin{proof}
        Suppose otherwise and put $\overline{H}=H/\operatorname{ann}_H(y_0)$. If $\operatorname{ann}_H(y_0)$ is of finite index, $y_0\in \widetilde{\operatorname{ann}}_G(H)$ that, being $y_0$ dimensional-generic, should imply that $\widetilde{\operatorname{ann}}_G(H)$ is of finite index in $G$. Consequently, $\lambda$ is virtually almost trivial.\\
        Assume, for a contradiction, that $\operatorname{ann}_H(y_0)\sim \operatorname{ann}_H(y_i)$. Then $\operatorname{ann}_H(y_0)\sim \operatorname{ann}_H(y_i)$ for every $i$ by indiscernibility. This implies that $\mathrm{im}(\lambda_{y_i})\sim \mathrm{im}(\lambda_{y_j})$ since $\mathrm{im}(\lambda_0)\apprle \mathrm{im}(\lambda_i)$ and 
        $$\operatorname{dim}(\mathrm{im}(\lambda_0))=\operatorname{dim}(H)-\operatorname{dim}(\operatorname{ann}_H(y_0))=\operatorname{dim}(H)-\operatorname{dim}(\operatorname{ann}_H(y_i))=\operatorname{dim}(\mathrm{im}(\lambda_i)).$$
        Let $R$ be the ring of quasi-endomorphisms of $\overline{H}$ modulo equivalence. Then, every element of $R$ is nilpotent or invertible by Lemma \ref{NilpoIsom}. The subset $I=\{r\in R:\ r \text{ is nilpotent}\}$ is an ideal in $R$. Indeed, it is clearly invariant for left and right multiplication. On the other hand, if $f,g\in I$, then $f+g$ is nilpotent. If not, there is $h\in R$ such that $h(f+g)=1$. But $(id-hg)(id+hg+(hg)^2+\cdot\cdot\cdot)=id$ (this sum is finite since $hg$ is nilpotent) implies that $id-hg=hf$ is invertible, a contradiction. Therefore, $R/I$ is a division ring. Let $\overline{r}$ be a finite tuple in $R/I$ definable over a finite set $A$. Then, by $\omega$-categoricity, $\overline{r}$ generates a finite subring of $R$. Therefore, by little Weddleburn's theorem, this is a locally finite field. Any quasi-endomorphism $\lambda_{y_i,y_j}$ is a quasi-automorphism on $\overline{H}$. Indeed, $\operatorname{ker}(\lambda_y)=\operatorname{ker}(\lambda_{y'})$ and therefore $\lambda_y$ is a quasi-isomorphism from $\overline{H}$ to $\mathrm{im}(\lambda_0)$. This implies that the composition of $\lambda_y$ and $\lambda_{y'}^{-1}$ is a quasi-automorphism. In particular, it is not nilpotent. By locally finiteness and indiscernibility, $\lambda_{y,y'}$ in $R/I$ has finite order equal to $N$. Since there are only finitely many $N$-roots of the unity, there are finitely many couples $\{(y_i,y_j)\}_{(i,j)\in \mathbb{Q}^2}$ such that, for any $y,y'\in \{y_i\}_{i\in \mathbb{Q}}$, there exists a couple $(y_i,y_j)$ such that $\lambda_{y_i,y_j}\sim\lambda_{y,y'}$. Therefore, by indescernibility, we may assume that all $\lambda_{y_i,y_j}$ are equal to the same $\eta+I$. Then, $\eta^2+I=\lambda_{y_j,y_k}\lambda_{y_i,y_j}+I=\eta+I$ and so $\eta=1$ (since $\eta\not=0$).\\
         Since $\lambda_{x_1-x_3,x_2-x_3}-id_H\in I$ then $B:=\mathrm{im}(\lambda_{x_1-x_3,x_2-x_3}-id_H)$ is of infinite index in $H$ almost containing $\operatorname{ann}_H(y_0)$.\\
         Consequently, for any $h\in H^{\circ}_{x_1,x_2,x_3}\leq \operatorname{dom}(\lambda_{x_1-x_3,x_2-x_3}-\operatorname{id}_H)$, there exists $b\in B$ such that $h+b\in \lambda_{x_1-x_3,x_2-x_3}(h)$. Hence,
         $$\lambda(x_1-x_3,h)=\lambda(x_2-x_3,h+b)=\lambda(x_2-x_3,h)+\lambda(x_2-x_3,b)$$
         whence 
         $$\lambda(x_1-x_2,h)=\lambda(x_1-x_3,h)-\lambda(x_2-x_3,h)=\lambda(x_2-x_3,b).$$
         This means that $\mathrm{im}(\lambda_{x_2-x_3})\simeq \mathrm{im}(\lambda_{x_1-x_2})$ is almost contained in $\lambda_{x_2-x_3}[B]$. But $\operatorname{ker}(\lambda_{x_2-x_3})\apprle B$ so 
         $$\operatorname{dim}(\lambda_{x_2-x_3}[B])=\operatorname{dim}(B)+\operatorname{dim}(\operatorname{ker}(\lambda_{x_2-x_3}))<\operatorname{dim}(\lambda_{x_2-x_3}).$$
         This is clearly a contradiction.
    \end{proof}
   We now complete the proof.\\
   Observe that $\mathrm{im}(\lambda_{y_i})$ is of infinite index in $K$ for $i\in \mathbb{Q}$. If not, there exists $n<\omega$ such that the following formula is satisfied by $(x_i,x_{i+\epsilon})$:
        $$\exists k_1,...,k_n\in K\ \forall k\in K\ \exists h\in H\ \bigvee_{i=1}^n k_i\lambda(x_i-x_{i+\epsilon},h)=k.$$
   By indiscernibility, every $y_j$ realizes this formula. Therefore, the images of $\lambda_{y_j}$ are all of finite index in $K$, implying that $\operatorname{ann}_{H}(y_j)\sim \operatorname{ann}_H(y_i)$ for all $i,j\in \mathbb{Q}$, a contradiction by the previous claim.\\
       Take $0<k\in \mathbb{Q}$ and put 
        $$A:=\{g\in G:\ \mathrm{im}(\lambda_g)\apprle \mathrm{im}(\lambda_{y_k})\}$$
        and
        $$H':=\{h\in H:\ \lambda'_h[A]\apprle \mathrm{im}(\lambda_{y_k})\}.$$ 
        We verify that $\operatorname{dim}(H')=\operatorname{dim}(H)$ and so $H'\sim H$. Let 
        $$X:=\{(a,h)\in A\times H : \lambda(a,h)\in \mathrm{im}(\lambda_{y_k})\}.$$
        $X$ is clearly definable. We prove that $\operatorname{dim}(X)=\operatorname{dim}(A\times H)$. Taken the first projection $\pi_1$, this is a surjective map from $X$ to $A$ with fibers $\pi^{-1}_1(a)=\{h\in H:\lambda(a,h)\in \mathrm{im}(\lambda_{y_k})\}$ that is of dimension equal to $\operatorname{dim}(H)$ by the assumption on $A$. By fibration, $\operatorname{dim}(X)=\operatorname{dim}(H)+\operatorname{dim}(A)$. Taken the second projection $\pi_2$, the image is $H$ and the fibers are $\pi_2^{-1}(h)=\{a\in A:\lambda(a,h)\in \mathrm{im}(\lambda_{y_k})\}$. $H'$ is clearly the set (definable by Lemma \ref{boundedind}) of $h\in H$ with fibers of maximal dimension. By union, $\operatorname{dim}(X)\leq \max\{\operatorname{dim}(H')+\operatorname{dim}(A),\operatorname{dim}(H-H')+\operatorname{dim}(A)-1\}$. $\operatorname{dim}(X)=\operatorname{dim}(A)+\operatorname{dim}(H)$ and necessarily $\operatorname{dim}(H')=\operatorname{dim}(H)$.\\
        We now study the bilinear quasi-form 
        $$\widetilde{\lambda}:A\times H'\to \mathrm{im}(\lambda_{y_k}).$$
        This has reduced dimension strictly less than $\operatorname{rdim}(\lambda)$. Indeed, $\operatorname{rdim}_{\widetilde{\lambda}}(A)\leq \operatorname{rdim}_{\lambda}(G)$ by Lemma \ref{induced}. The same holds for $\operatorname{rdim}_{\widetilde{\lambda}}(H')$. Finally, $\operatorname{dim}(\lambda_{y_i})<\operatorname{dim}(K)$ by previous proof.
        By assumption on the minimality of $\operatorname{dim}(\lambda)$, $\widetilde{\lambda}$ is virtually almost trivial and so $\widetilde{\operatorname{ann}}_{A}(H')$ is a $(x_k,x_{k+\epsilon})$-definable subgroup of finite index in $A$ such that, for all $g\in \widetilde{\operatorname{ann}}_{A}(H')$, $\lambda(g,\widetilde{\operatorname{ann}}_{H'}(A))$ is finite. Since $\widetilde{\operatorname{ann}}_{H'}(A)$ is of finite index in $H'$ and so also in $H$, $\widetilde{\operatorname{ann}}_{A}(H')\leq \widetilde{\operatorname{ann}}_G(H)$. Taken $y_l$ for $l<k$, $\mathrm{im}(y_l)\apprle \mathrm{im}(y_k)$ by Claim $1$ and so $y_l\in A$. By principality, $y_l\in A^{\circ}_{y_k}$. This implies that $y_l\in \widetilde{\operatorname{ann}}_G(H)$. By indiscernibility, this holds for every $y_i$. Since $y_0$ is dimensional-generic, this implies that $\widetilde{\operatorname{ann}}_G(H)$ is of finite index in $G$. Finally, the proof follows from Lemma \ref{almosttrivial}.
    \end{proof}
\section{Groups and rings}
We show that finite-dimensional $\omega$-categorical groups are finite-by-abelian-by-finite and that finite-dimensional $\omega$-categorical rings are virtually finite-by-null.\\
We start by verifying that the case (2) in Theorem \ref{CharSimGro} is impossible for finite-dimensional groups. In other words, we need to prove that a definable characteristically simple group $G$ of finite dimension cannot be isomorphic to $B(F)$ or $B^-(F)$. This is a consequence of the following Lemma.
\begin{lemma}\label{B(F)}
    The groups $B(F)$ and $B^-(F)$ are not of finite dimension \hbox{i.e.} there is no finite-dimensional theory $T$ in which $B(F)$ and $B^-(F)$ can be defined.
\end{lemma}
\begin{proof}
    It is sufficient to observe that there exists an infinite chain of definable subgroups $\{C_i\}_{i<\omega}$ such that $|C_i/C_{i+1}|$ is infinite, contradicting $\omega$-DCC.\\
    Let $\{A_i\}_{i<\omega}$ be a strictly ascending chain of clopen subsets in the Cantor set (this clearly exists). We take, for $f$ a non central element in $F$, the function $g_1$ such that $g_1(A_1)=\{f\}$ and $g_1(C-A_1)=\{0\}$. Let $C_1$ be the centraliser in $B(F)$ of $g_1$. This contains all the functions $g$ such that $g(A_1)\leq C_F(f)$ and $g$ assumes arbitrary values outside $A_1$. Let $g_2$ be a function with $g_2(A_2)=f$ and let $C_2$ be the centraliser of $g_2$ in $B(F)$. We verify that $|C_2:C_1|$ is infinite. Indeed, there exist pairwise disjoint clopen subsets $\{D_j\}_{j<\omega}$ of $C$ such that $D_j\subseteq A_2$ and $D_j\cap A_1=\emptyset$. Take the functions $s_j\in B(F)$ such that $s_i(A_1)=f$ and $s_i(D_i)=d_i\in F-C_F(f)$. Then, $s_j^{-1}s_i\in C_1$ but not in $C_2$. This verifies that the index of $C_2$ in $C_1$ is infinite. Iterating the proof for any $i$, we obtain the contradiction. The proof for $B^-(F)$ is similar.
\end{proof}
The second step is to show that a $\omega$-categorical finite-dimensional definable group is virtually nilpotent. To do this, we can proceed in two ways: either as in \cite{dobrowolski2020omega} or using the two following results.
\begin{lemma}
    A locally nilpotent definable group of finite dimension is virtually nilpotent.
\end{lemma}
\begin{proof}
    It is sufficient to prove that, given $G$ a locally nilpotent definable group of finite dimension, $\widetilde{Z}(G)$ is infinite. Iterating the proof, we conclude that $G$ is almost nilpotent. Consequently, $G$ is virtually nilpotent from \cite[Proposition 3.29]{hempel2020almost}.\\
    Assume, by contradiction, that the almost center is finite. Then $G/\widetilde{Z}(G)$ is a definable locally nilpotent group of finite dimension with trivial almost center by Lemma \ref{Zfin}. We verify that this is impossible. Let $C$ be a centraliser of a tuple of minimal dimension between all the centralisers. If $\dim(C)=0$, take $C_1$ a minimal finite centraliser of a finite tuple. Then $C_1\cap C_{G}(g)$ is not empty by locally nilpotency. By minimality, $C_1\leq C_{G}(g)$ for all $g\in G$ \hbox{i.e.} $C_1\leq Z(G)=\{0\}$, a contradiction. Assume that $C$ is infinite. Then, $C\cap C_{G}(g)$ has same dimension as $C$ and so $C\apprle C_{G}(g)$ for all $g\in G$ \hbox{i.e.} $G\leq \widetilde{C}_G(C)$. By Lemma \ref{sym}, $C\apprle \widetilde{Z}(G)$. This is again contradictory.
 \end{proof}
Since an $\omega$-categorical group of exponent $p$ is locally nilpotent, the third case of Lemma \ref{CharSimGro} is impossible. Since the second case is excluded by Lemma \ref{B(F)}, we have the following conclusion.
\begin{lemma}\label{AbeChaSim}
    Let $G$ be an $\omega$-categorical finite-dimensional characteristically simple group. Then, $G$ is an elementary abelian $p$-group.
\end{lemma}
By Lemma \ref{ChaSimSer} and Lemma \ref{AbeChaSim}, for any $\omega$-categorical group $G$ of finite dimension, there exists a series of definable characteristic subgroups $\{G_i\}_{i\leq n}$ such that:
    \begin{itemize}
        \item $G_i/G_{i+1}$ is characteristically simple;
        \item $G_i/G_{i+1}$ is an elementary abelian $p$-group.
    \end{itemize}
Therefore, we have the following conclusion.
\begin{lemma}
Let $G$ be a soluble finite-dimensional group. Then, $G$ is soluble and of finite exponent.
\end{lemma}
Finally, the virtual nilpotency follows from the next Lemma.
\begin{lemma}
    Let $G$ be a finite-dimensional $\omega$-categorical group. Then, $G$ is virtually nilpotent.
\end{lemma}
\begin{proof}
By \cite[Proposition 3.29]{hempel2020almost}, it is sufficient to prove that $G$ is almost nilpotent. We prove the Lemma by induction on the dimension. Clearly, if $G$ has infinite almost center, the conclusion follows by induction hypothesis. Let $\{G_i\}_{i\leq n}$ be a series of definable subgroups of $G$ such that $G_i/G_{i+1}$ is infinite and almost abelian. This exists by \cite[Proposition 3.21]{hempel2020almost}. Clearly, the almost abelian group $G_{n-2}/G_{n-1}$ acts on the abelian group $G_{n-1}/G'_{n-1}$. Let $A$ be a definable subgroup of $G_{n-1}/G'_{n-1}$ of minimal dimension between all the subgroup that are minimal for a definable subgroup of finite index in $G_{n-2}/G_{n-1}$. Let $H\leq G_{n-2}/G_{n-1}$ be a definable subgroup of finite index such that $A$ is $H$-invariant. Then, by construction, $A$ is absolutely $H$-minimal. By Linearization Theorem \ref{LinInv}, if the action is not almost trivial, there exists a finite subgroup $F_1$ in $A$ and a definable subgroup $F_2$ not of finite index in $H$ such that $A/F_1\simeq R^+$ and $H/F_2$ embeds in $R^{\times}$. The latter contradicts the fact that $H$ is of finite exponent. Therefore, the action is almost trivial. This implies that $\widetilde{C}_{G_{n-1}}(G_{n-2})=:C$ is a definable infinite $G$-normal subgroup in $G_{n-1}$. By Lemma \ref{sym}, $\widetilde{C}_{G_{n-2}}(C)$ is of finite index in $G_{n-2}$. Moreover, by Lemma \ref{lemma 4}, $[C,\widetilde{C}_{G_{n-2}}(G)]=:C_1$ is a finite $G$-normal subgroup. $G_{n-3}/\widetilde{C}_{G_{n-2}}$ is almost abelian and acts on $C/C_1$. Iterating the proof for any $G_i$, we obtain that that $\widetilde{Z}(G)$ is infinite.   
\end{proof}
We can also verify that $G$ is virtually nilpotent-by-finite with a proof similar to \cite{dobrowolski2020omega}. We need the following lemma.
\begin{lemma}\label{BooAlg}
    An atomless Boolean algebra cannot have finite dimension (\hbox{i.e.} it cannot be defined in a finite-dimensional theory).
\end{lemma}
\begin{proof}
    Assume, for a contradiction, that there exists a finite-dimensional theory $T$ and $B$ an atomless Boolean algebra definable in $T$. Assume $B$ of minimal dimension $n$ among all the infinite atomless Boolean algebras definable in $T$.\\
    Let $a\in B$ be an element not equal to $0$ or $1$ and define the function
    $$f_a: (B, \vee)\to  (B_a,\vee)$$
    that sends $b$ in $f_a(b)=a\wedge b$ where $B_a=\{b\in B:\ \ b\leq a\}=\{b\in B:\ \exists c\ a\wedge c=b\}$. This map is clearly surjective. We verify that $B_a$ is an atomless Boolean algebra. Since $B_a$ is clearly definable, it is either finite or its dimension is equal to $n$. It is closed for $\wedge$ and $\vee$ since $a\wedge (b\vee c)=a\wedge b\vee a\wedge c$ and $a\wedge (a\wedge b)=a\wedge b$ (since $a\wedge 1=a$, this also proves that $f_a$ is an homorphism of monoids). Therefore, $B_a$ is closed for $\vee$. It is closed also for $\wedge$ since $a\wedge b\wedge a\wedge c=a\wedge (b\wedge c)$. The other properties of Boolean algebra follow easily (assuming $1_{B_a}=a$). Finally, we prove that it is atomless. It is sufficient to show that, for any $a\wedge c\in B_a$, there exists $d\in B_a$ such that $d\not=0,a\wedge c$ and $d< a\wedge c$. Since $B$ is an atomless Boolean algebra and $a\wedge c\in B$, there exists $d\in B$ such that $0<d< a\wedge c\leq a$ and so $d\in B_a$. Therefore, $n=\operatorname{dim}(B)\geq \operatorname{dim}(B_a)\geq n$ so $\operatorname{dim}(B_a)=n$. If we verify that $f_a^{-1}(c)$ is infinite for every $c\in B_a$, it implies that $n=\operatorname{dim}(B)\geq 1+\operatorname{dim}(B_a)=1+n$ by fibration, a contradiction. We prove that $f_a^{-1}(0)=\{b\in B:\ b\leq \lnot a\}$. The latter set is infinite since $B$ is atomless and $\lnot a=0$ iff $a=1$. Assume $b\leq \lnot a$. Then, $a\wedge b=a\wedge( \lnot a\wedge b)=(a\wedge \lnot a)\wedge b=0\wedge b=0$. On the other hand, if $a\wedge b=0$, then $b=1\wedge b=(a\vee \lnot a)\wedge b=(a\wedge b)\vee (\lnot a\wedge b)=\lnot a\wedge b$ and so $b\leq \lnot a$.\\
    Given $b\in f_a^{-1}(0)$ and $d\in f_a^{-1}(c)$, then $a\wedge (b\vee d)=(a\wedge b)\vee (a\wedge d)=a\wedge d=c$. Clearly, $f_a(d)=f_a(a\wedge d)$. Therefore, any element $(a\wedge d)\vee b$ with $b\in f_a^{-1}(0)$ is in $f_a^{-1}(c)$. Assume that, for $b\not=b_1\leq \lnot a$, $(a\wedge d)\vee b=(a\wedge d)\vee b_1$. Then, $\lnot a\wedge ((a\wedge d)\vee b)=\lnot a\wedge ((a\wedge d)\vee b_1)$. But the first is equal to $\lnot a\wedge b$ and the second to $\lnot a\wedge b_1$. Since $b,b_1\leq \lnot a$, these are equal to $b$ and $b_1$, contradicting the assumption that $b\not=b_1$. Therefore, each fiber is infinite, and we reach a contradiction.   
\end{proof}
Since any $\omega$-categorical group can be written as a chain of characteristic subgroups with characteristically simple factors, we may assume that every $\omega$-categorical finite-dimensional subgroup is soluble. If it is not nilpotent-by-finite, an atomless Boolean algebra can be defined on it by \cite{archer1997soluble}, contradicting Lemma \ref{BooAlg}.\\
Therefore, we may assume that every $\omega$-categorical group of finite dimension is nilpotent-by-finite. From \cite[Proposition 3.29]{hempel2020almost}, any $\omega$-categorical finite-dimensional subgroup has a characteristic nilpotent subgroup of finite index. Finally, we prove that any nilpotent $\omega$-categorical subgroup of finite dimension is finite-by-abelian-by-finite. This verifies the meta-conjecture for the group case.
\begin{lemma}
    A nilpotent $\omega$-categorical group of finite dimension is finite-by-abelian-by-finite.
\end{lemma}
\begin{proof}
    We may take a counterexample of minimal nilpotent class possible.\\
    By minimality, $G/Z(G)$ is finite-by-abelian-by-finite. Since $G$ is not finite-by-abelian-by-finite, $Z(G)$ cannot be finite. The pre-image of $\widetilde{Z}(G/Z(G))$ is a characteristic subgroup of finite index in $G$, and so we may replace $G$ with it. By \cite[Proposition 3.27]{hempel2020almost}, $F:=G'Z(G)$ is a finite extension of $Z(G)$. Again by \cite[Proposition 3.27]{hempel2020almost}, $F'$ is finite and characteristic. Therefore, the quotient $G/F'$ is again $\omega$-categorical and we can assume $F'=\{0\}$. As $F/Z(G)$ is characteristic, $C_G(F/Z(G))$ is a characteristic definable subgroup of finite index. Up to take $G=C_G(F/Z(G))$, we may assume $[G,F]\leq Z(G)$ \hbox{i.e.} $F\leq Z^2(G)$. Consequently, the map 
    $$[\_,g]:G\to F/Z(G)$$
    that sends $g'$ in $[g,g']+Z(G)$ is a definable homomorphism. The kernel of $H_g$ has finite index in $G$ for any $g$. Consequently, the map
    $$[\_,g]:h\in H_g\to [h,g]\in Z(G)$$
    is a homomorphism with kernel containing $H_g'Z(G)$, since the image is abelian. Let $G_0$ be a definable subgroup in $G$ of finite index such that $|G_0':G'_0\cap Z(G)|$ is minimal. Then, for any definable subgroup $H$ of finite index, $G'_0Z(G)$ and $(H\cap G_0)'Z(G)$ are finite extensions of $Z(G)$ with $G'_0Z(G)\geq (G_0\cap H)'Z(G)$. Moreover, $|G'_0Z(G)/Z(G)|=|G'_0/Z(G)\cap G'_0|$ and $|(G_0\cap H)'Z(G):Z(G)|=|(G_0\cap H)'/Z(G)\cap (G_0\cap H)'|$. By minimality, the latter must be equal to the first. Therefore, $G'_0Z(G)=(H_g\cap G_0)'Z(G)\leq H'_gZ(G)\leq C_G(g)$. This implies that $G'_0\leq Z(G)$ and so $Z_2(G)$ is of finite index in $G$. Since $Z_2(G)$ is a characteristic subgroup, we may assume that $G$ is nilpotent of class $2$. Then the function
    $$[\_,\_]:(g_1,g_2)\in G\times G\to [g_1,g_2] \in Z(G)$$
    is a bilinear quasi-form. By Theorem \ref{VirtTriv}, it should be virtually almost trivial, and this is equivalent to saying that $G$ is finite-by-abelian-by-finite.
\end{proof}
In conclusion, we have proven the following theorem.
\begin{theorem}
    Let $G$ be an $\omega$-categorical group of finite dimension. Then $G$ is finite-by-abelian-by-finite.
\end{theorem}
For rings, the proof is a straightforward corollary of Theorem \ref{VirtTriv}.
\begin{lemma}
    Let $R$ be an $\omega$-categorical ring of finite dimension. Then, $R$ is virtually finite-by-null.
\end{lemma}
\begin{proof}
    It is sufficient to observe that the function
    $$\_\cdot\_:R\times R\to R$$
    that sends $(r,r')$ in $r\cdot r'$ is a bilinear form. By Theorem \ref{VirtTriv}, it must be virtually almost trivial. Therefore, $R$ is virtually finite-by-null.
\end{proof}

\end{document}